\documentclass{article}

\usepackage{a4}
\usepackage{psfig}
\usepackage{fancyhdr}
\usepackage{latexsym}
\usepackage{amssymb}
\usepackage{mathtools}

\usepackage{abstract}

\usepackage{amstext}
\usepackage{amsfonts}
\usepackage{mathrsfs}
\usepackage{multirow}
\usepackage{multicol}
\usepackage{stmaryrd}
\usepackage{stackrel}

\usepackage[dvips]{epsfig}
\usepackage[latin1]{inputenc}

\usepackage{csquotes}

\usepackage[square]{natbib}
\setcitestyle{numbers}

\usepackage{float}
\usepackage{verbatim}
\usepackage{longtable}
\usepackage{tabularx}       
\usepackage{multicol}
\usepackage{color}
\usepackage{wrapfig}
\usepackage{floatflt}
\usepackage{afterpage}

\usepackage{textcomp}
\usepackage[gen]{eurosym}

\usepackage{amsmath}
\usepackage{mathabx}
\usepackage{amssymb}
\usepackage{amstext}
\usepackage{mathrsfs}
\usepackage{amsfonts}
\usepackage{bbm}

\usepackage{graphicx}
\usepackage{epsfig}
\usepackage{subfigure}

\usepackage{nomencl}
\usepackage{epigraph}
\makenomenclature

\usepackage{url}
\usepackage{appendix}

\usepackage{amsthm}
\usepackage{yfonts}

\newcommand{\F}{\mathbb{F}}

\newcommand{\R}{\mathbb{R}}

\newcommand{\PP}{\mathbb{P}}
\newcommand{\EE}{\mathbb{E}}
\newcommand{\dd}{\mathrm{d}}

\newcommand{\AY}{Az\'{e}ma-Yor }

\newcommand{\indicator}[1]{\mathbbm{1}_{\left\{ {#1} \right\} }}
\newcommand{\E}[1]{\mathbb{E}\left[ {#1} \right] }
\newcommand{\Prob}[1]{\mathbb{P}\left[ {#1} \right] }

\newcommand{\bb}{\boldsymbol{b}}

\newcommand\restr[2]{{
  \left.\kern-\nulldelimiterspace 
  #1 
  \vphantom{\big|} 
  \right|_{#2} 
  }}

\theoremstyle{plain}
\newtheorem{Theorem}{Theorem}[section]
\newtheorem{Assumption}{Assumption}

\theoremstyle{definition}
\newtheorem{Definition}[Theorem]{Definition}

\theoremstyle{remark}
\newtheorem{Remark}[Theorem]{Remark}

\numberwithin{equation}{section}
\numberwithin{figure}{section}

\usepackage[
bookmarks,
bookmarksopen=true,
pdftitle={},
pdfauthor={Peter Spoida},
pdfcreator={Peter Spoida},
pdfsubject={},
pdfkeywords={},
colorlinks=false,
anchorcolor=black,
filecolor=magenta, 
menucolor=red, 
linkcolor=black, 
plainpages=false,
pdfpagelabels,
hypertexnames=false,
linktocpage 
]{hyperref}

\title{ 
Characterization of Market Models in the Presence of Traded Vanilla and Barrier Options}
\author{
Peter Spoida \thanks{
The author thanks Samuel Cohen, Alexander Cox, Marek Musiela and Jan Ob\l{}\'{o}j for insightful discussions.
\newline 
\href{mailto:peter.spoida@gmail.com}{peter.spoida@gmail.com} 
}
}

\begin{document}
\date{\today}
\maketitle

\begin{abstract}

We characterize the set of market models when there are a finite number of traded Vanilla and Barrier options with maturity $T$ written on the asset $S$.
From a probabilistic perspective, our result describes the set of joint distributions for $(S_T, \sup_{u \leq T} S_u)$ when a finite number of marginal law constraints on both $S_T$ and $\sup_{u \leq T} S_u$ is imposed.
An extension to the case of multiple maturities is obtained.

Our characterization requires a decomposition of the call price function and once it is obtained, we can explicitly express certain joint probabilities in this model. 
In order to obtain a fully specified joint distribution we discuss interpolation methods.




\end{abstract}

\section{Introduction}

Calibration of models to market data is one major challenge in Mathematical Finance.
Typically, one uses call option prices to incorporate information about univariate distributional properties.
In some markets there are in addition options which are informative about joint distributional properties.
Here we are interested in the case when there are traded Vanilla and Barrier options.
We obtain a characterization of these joint distributions for the stock and its running maximum in the case of traded options with multiple maturities.
Our characterization requires a kind of decomposition of certain call price functions and once it is obtained, we have an explicit expression for certain joint probabilities in the models characterized by this decomposition.
We discuss interpolations of these joint probabilities which yield a fully specified marginal joint distribution which is consistent with the market.

Once one is given marginal joint distributions, we note that there are methods for defining diffusion-type models consistent with these marginals, cf. research by \citet{CarrGamma}, \citet{CHO11timehomodiff}, \citet{RePEc:eee:spapps:v:121:y:2011:i:12:p:2802-2817} and \citet{Forde13}.

Our method of proof relies on theory related to the Skorokhod embedding problem, see \citet{Obloj:04b} for a survey.

The joint laws of a martingale and its maximum have been characterized by \citet{Kertz90} and \citet{Rogers:93}.


\paragraph*{Notation}
The underlying asset will be denoted by $S$. For its maximum we write $M_T = \sup_{t \leq T} S_t$. 
The standard Markovian time-shift operator is denoted by $\theta_t(\omega):= ( \omega_{u} )_{u \geq t}$.
The first hitting time of $B$ is denoted by $H_B$.

\section{Characterization of Market Models}
\label{sec:Characterization of Market Models}

In this section we present our main results. 
We characterize the existence of a market model, both in the case of a single and multiple maturities.

\subsection{Market Data and Market Models}

Suppose $c({K_1}), \dots, c({K_n})$ are the prices for call options with strikes $0 < K_1 < \dots < K_n$, respectively.
Further, let $\bb = \big( b(B_1), \dots, b(B_m) \big)$ be the prices for simple barrier options $\indicator{S_T \geq B_j}$ with barrier levels $S_0 =: B_0 < B_1 \leq \dots \leq B_m$, respectively.

\begin{Definition}[Market Model]
\label{def:market_model}
A market model is a filtered probability space $(\Omega, \mathcal{F}, \F, \PP)$ where the filtration $\F=(\mathcal{F}_t)_{t \in [0,T]}$ satisfies the usual hypothesis and there is an $\F$-adapted martingale $S$ defined on this space satisfying
\begin{subequations}
\begin{alignat}{3}
\E{ \left( S_T - K_i \right)^+ } &= c(K_i) , \qquad &&i=1,\dots,n, \label{eq:market_model_1}\\
\E{ \indicator{ M_T \geq B_j } } &= b(B_j), \qquad &&j=1,\dots,m.
\label{eq:market_model_2}
\end{alignat}
\end{subequations}


Using this definition, it is clear how to extend the notion of market model to a setting where there are options with multiple maturities.
\end{Definition}


For our main result we require the following assumption.
\begin{Assumption}[Asset]
\label{ass:asset and market}
Assume that the asset $S$ 
\begin{itemize}
	\item[(a)] has continuous trajectories,
	\item[(b)] has zero cost of carry (e.g. when interest rates are zero),
	\item[(c)] is (strictly) positive.
\end{itemize}
\end{Assumption}


\subsection{One Maturity}

We now state and prove our main result in the case of one maturity.
This result will be extended to multiple maturities in the next section. 
The proof of this extension will be an induction over the number of maturities and hence will rely on the one maturity statement. 

\begin{Theorem}[Characterization Market Model -- One Maturity]
\label{thm:Characterization Market Model - One Maturity}
Let Assumption \ref{ass:asset and market} hold.
Then there exists a market model if and only if 
\begin{itemize}
	\item[1.] there exists a call price function\footnote{i.e. a non-negative convex function $c$ such that $-\frac{\partial c}{\partial x}(0+) \leq 1$ and $c(x) \to 0$ as $x \to \infty$, cf. \citet{DavisHobson07}.} $c_{\mu}(\cdot) = \int ( x-\cdot )^+ \mu(\dd x)$ which interpolates the given call prices $c(K_1),\dots,c(K_n), \ -c'_{\mu}(0+) = 1$.
	\item[2.] there exist $\mathfrak{c}^{B_1},\dots , \mathfrak{c}^{B_m} $ such that for all $j=1,\dots,m$,
\begin{subequations}
\begin{alignat}{3}
&\mathfrak{c}^{B_j} : \R_{\geq 0} \to [0,S_0]   \qquad && \text{is convex,} \label{eq:cond_1} \\
&0 \leq \mathfrak{c}^{B_1} \leq  \dots \leq \mathfrak{c}^{B_m} \leq \mathfrak{c}^{B_{m+1}} := c_{\mu}, && 
\label{eq:cond_2} \\
-&\frac{\dd \mathfrak{c}^{B_j}}{\dd x }(0+) = 1, \quad \mathfrak{c}^{B_j}(0) = S_0, \quad \mathfrak{c}^{B_j}(x) = 0, \qquad && \forall x \geq B_j, 
\label{eq:cond_3} \\
-&\frac{\dd \mathfrak{c}^{B_j}}{\dd x }(B_j-) = b(B_j),  && 
\label{eq:cond_4}\\
&x \mapsto \mathfrak{c}^{B_{j+1}}(x) - \mathfrak{c}^{B_j}(x) +  b(B_j)  ( B_j - x )^+ \qquad &&\text{is convex.} \label{eq:cond_5}
\end{alignat}
\end{subequations}
\end{itemize}

This market model can be chosen with bounded support (but does not have to).

Furthermore, in the market models characterized by \eqref{eq:cond_1}--\eqref{eq:cond_5} we have for $0 \leq x < B_j$,
\begin{align}
\Prob{ S_T > x, M_T < B_j } = -\frac{\dd \mathfrak{c}^{B_j}}{\dd x}(x+) -b(B_j).
\label{eq:joint_prob}
\end{align}
\end{Theorem}

\begin{proof}
Let $(\Omega, \mathcal{F}, \F, \PP)$ be a market model, i.e.
\begin{alignat*}{5}
\E{ \left( S_T - K_i \right)^+ } &= c(K_i) \qquad &&\forall && i&&=1,\dots,n, \\
\Prob{ M_T \geq B_j } &= b(B_j) \qquad &&\forall && j&&=1,\dots,m.
\end{alignat*}
By continuity of $S$ the Dambis-Dubins-Schwarz time change yields
\begin{align*}
\left( S_t \right)_{t \leq T} = \left( X_{\rho_t} \right)_{t \leq T}
\end{align*}
where $X$ is Geometric Brownian motion started at $S_0$ and $\rho_t = \langle \log(S) \rangle_t$.

Define
\begin{align}
\mathfrak{c}^{B_j}(x) := \E{ \left( X_{\rho_T \wedge H_{B_j}} - x \right)^+ } \qquad j=1,\dots,m, \ \ \ x \in \R.
\end{align}
Clearly $\mathfrak{c}^{B_1}, \dots, \mathfrak{c}^{B_m}$ satisfy \eqref{eq:cond_1}--\eqref{eq:cond_3}.

Note 
\begin{align*}
b(B_j) = \Prob{ M_T \geq B_j } = \Prob{ X_{ \rho_T \wedge H_{B_j} } \geq B_j } = -\frac{\dd \mathfrak{c}^{B_j}}{\dd x }(B_j-)
\end{align*}
which is condition \eqref{eq:cond_4}.
Note also 
\begin{align*}
b(B_j) = \Prob{ \rho_T \geq H_{B_j} }.
\end{align*}
Let $X^{B_j}$ denote a Geometric Brownian motion started at $B_j$.
We have for $j<m$,
\begin{alignat}{3}
& && \mathfrak{c}^{B_{j+1}}(x) - \mathfrak{c}^{B_j}(x) 
\nonumber \\
=& && \E{ \indicator{ \rho_T < H_{B_j } } \left( X_{\rho_T \wedge H_{B_j} } - x \right)^+ + \indicator{ \rho_T \geq H_{B_j} } \left( X_{\rho_T \wedge H_{B_j} } - x \right)^+    } \nonumber \\
&-&& \Prob{ \rho_T \geq H_{B_j} } \cdot ( B_j - x )^+  - \E{ \left( X_{\rho_T \wedge H_{B_j} } - x \right)^+ } 
\nonumber  \\
&+&& \E{ \left( X_{\rho_T \wedge H_{B_{j+1}}} - x \right)^+ \indicator{ \rho_T  \geq H_{B_j} } } 
\nonumber \\
=& &&\Prob{ \rho_T \geq H_{B_j} } \cdot \left( -\left( B_j - x \right)^+ + \E{ \left( X^{B_j}_{{\psi}} - x \right)^+ } \right) 
\label{eq:condition_5_necessity}  
\end{alignat}
where $\psi$ is an independent stopping time (independent of $X$) such that 
\begin{align*}
\Prob{ \psi \leq x } &= \Prob{ (\rho_T \wedge H_{B_{j+1}} )- H_{B_j} \leq x \big | \rho_T \geq H_{B_j}}.
\end{align*}
These two properties of $\psi$ together with the Markov property of $X$ ensure that $\psi$ embeds the same distribution as required by the above equation.
Then, since $x \mapsto \E{ \left( X^{B_j}_{\psi} - x \right)^+ }$ is convex condition \eqref{eq:cond_5} follows. 

Now assume that conditions \eqref{eq:cond_1}--\eqref{eq:cond_5} hold.
By the Skorokhod embedding theorem the claim is true for $m=0$.
Inductively, let us assume that the claim is true until $m-1$, i.e. there exist stopping times $\gamma_{1} \leq \dots \leq \gamma_{m-1}$, $\gamma_j \leq H_{B_j}$, such that for $j = 1,\dots, m-1,$
\begin{align*}
\mathfrak{c}^{B_j}(x) &= \E{ \left( X_{\gamma_j} - x \right)^+ } \qquad \forall x \in \R, \\
b(B_j) &= \Prob{ \gamma_j = H_{B_j} }.
\end{align*}
Define
\begin{align}
\varphi_m(x) := \frac{ \mathfrak{c}^{B_m}(x)- \E{ \left( X_{ \gamma_{m-1} } - x \right)^+ } }{b(B_{m-1})} + \left( B_{m-1} - x \right)^+.
\label{eq:definition_varphi_m}
\end{align}
It follows by induction hypothesis and by \eqref{eq:cond_3} and \eqref{eq:cond_5} that $\varphi_m$ defines a call price function, i.e. $\varphi_m$ is convex, $\varphi_m(0) = B_{m-1}$, $-\varphi'_m(0+) = 1$ and $\varphi_m(x) = 0$ for $x \geq B_m$.
Hence, there exists a stopping time $\vartheta_m$ such that
\begin{align}
\varphi_m(x) = \E{ \left( X^{B_{m-1}}_{ \vartheta_m } -x \right)^+ } \qquad \forall x \in \R.
\label{eq:existence_vartheta_m}
\end{align} 
Therefore, by \eqref{eq:definition_varphi_m}, \eqref{eq:existence_vartheta_m} and induction hypothesis,
\begin{alignat}{3}
& \ && \mathfrak{c}^{B_m}(x) \nonumber \\
=&&& \Prob{ \gamma_{m-1} = H_{B_{m-1}} } \cdot \E{ \left( X^{B_{m-1}}_{ \vartheta_m }  -x \right)^+ } + \E{ \left( X_{ \gamma_{m-1} } - x \right)^+ \indicator{ \gamma_{m-1} < H_{B_{m-1}} } } \nonumber \\
=& && \E{ \left( X_{\gamma_m} - x \right)^+ }
\label{eq:existence_increment_embedding}
\end{alignat}
where
\begin{align}
\gamma_m := 
\begin{cases} 
	\gamma_{m-1} &\mbox{if } \gamma_{m-1} < H_{B_{m-1}}, \\
	\gamma_{m-1}+ \vartheta_m & \mbox{if } \gamma_{m-1} = H_{B_{m-1}}. 
\end{cases}
\end{align}
Clearly, $\gamma_m \leq H_{B_m}$ and
\begin{align*}
\Prob{ \gamma_m = H_{B_m} } = \Prob{ X_{\gamma_m} = B_m } = \Prob{ X_{\gamma_m} \geq B_m } = -\mathfrak{c}^{B_m}(B_m-) = b(B_m)
\end{align*}
and 
\begin{align}
\Prob{ \gamma_m \geq H_{B_j} } = \Prob{ \gamma_{m-1} \geq H_{B_j} } = b(B_j) \qquad \forall j \leq m-1.
\end{align}
By the same argument as above, there exists $\gamma_{m+1} \geq \gamma_m $ such that 
\begin{align*}
&\Prob{ \gamma_{m+1} \geq H_{B_j} } = \Prob{ \gamma_{m} \geq H_{B_j} } = b(B_j), \qquad j=1,\dots,m, \\
&\E{ \left( X_{\gamma_{m+1}} - x \right)^+ } = \mathfrak{c}^{B_{m+1}}(x) = c_{\mu}(x).
\end{align*}
Taking $S_t := X_{ \frac{t}{T-t} \wedge \gamma_{m+1} }$ yields a market model.

As for the bounded support claim, we note that choosing a sufficiently large upper bound for the support of $\mu$, will allow us to choose $c_{\mu} = \mathfrak{c}^{B_{m+1}}$ in a way to satisfy \eqref{eq:cond_5} for $j=m$.

Finally, using the notation from the proof, we get by rearranging \eqref{eq:existence_increment_embedding} and writing $\frac{\dd^+}{\dd x}$ for the right-derivative,
\begin{alignat*}{3}
& &&\Prob{ S_T > x, M_T < B_j } \\
= & \ - \ && \frac{\dd^+}{\dd x} \left( \E{ \left( X_{\gamma_j} - x \right)^+ \indicator{ \gamma_j < H_{B_j} } } \right) \\
=& \ - \ &&\frac{\dd^+}{\dd x} \left( \mathfrak{c}^{B_{j+1}}(x) - b(B_j) \left[ \frac{\mathfrak{c}^{B_{j+1}(x)} - \mathfrak{c}^{B_j}(x)}{b(B_j)} + (B_j - x)^+ \right] \right) \\
=& \ - \ &&\frac{\dd \mathfrak{c}^{B_j}}{\dd x}(x+) -b(B_j)
\end{alignat*}
for $0 \leq x < B_j$.
\end{proof}

\subsection{Multiple Maturities}

We now extend Theorem \ref{thm:Characterization Market Model - One Maturity} to the setup of multiple maturities.

For simplicity we assume that the strikes and barriers at each maturity coincide.
Denote the right endpoint of the support of the measure $\mu$ by $r_{\mu} := \inf \big \{ x:  \mu(( x,\infty )) = 0 \big\}$.

Take $0<T_1< \dots < T_k$.
Suppose $c_l({K_1}), \dots, c_l({K_n})$ are the prices for call options with strikes $0 < K_1 < \dots < K_n$ and maturity $T_l$, $l=1,\dots,k$.
Further, let $\bb_l = \big( b_l(B_1), \dots, b_l(B_m) \big)$ be the prices for simple barrier options with barrier levels $S_0 =: B_0 < B_1 \leq \dots \leq B_m$ and (deterministic) maturity $T_l$, $l=1,\dots,k$. 
Set $b_l(B_0) := 1$ and $b_0 \equiv 0$.

\begin{Theorem}[Characterization Market Model -- Multiple Maturities]
\label{thm:Characterization Market Model - Multiple Maturity}
Let Assumption \ref{ass:asset and market} hold. 
Then there exists a market model if and only if 
\begin{itemize}
	\item[1.] there exist call price functions $c_{\mu_l}(\cdot) = \int ( x-\cdot )^+ \mu_l(\dd x)$ which interpolate the given call prices for (deterministic) maturity $T_l$, $l=1,\dots, k$, and satisfy $c_{\mu_1} \leq \dots \leq c_{\mu_k}, \ -c'_{\mu_l}(0+) = 1$. 
	For $B_{m+1}:= r_{\mu_k}$ we set $b_l(B_{m+1}) := -\frac{\partial c_{\mu_l}}{\partial x}(B_{m+1}-)$.
	\item[2.] there exist $\left\{ c^{B_j}_{l} \right\}$ such that for all $j=1,\dots,m+1$ and $l=1,\dots,k$, 
\begin{subequations}
\begin{alignat}{3}
&c_l^{B_j} : \R_{ \geq 0} \to [0,S_0] \qquad  \text{is convex,} &&
\label{eq:cond_a} \\
-& \frac{\dd c_l^{B_j}}{\dd x }(0+) = b_l(B_{j-1}), \quad  c_l^{B_j}(0) = b_l(B_{j-1}) \cdot B_{j-1}, &&
\label{eq:cond_b} \\
-&\frac{\dd c_l^{B_j}}{\dd x }(B_j-) = b_l(B_j), \quad c_l^{B_j}(x) = 0,    \qquad && x \geq B_j, 
\label{eq:cond_c} \\
& c_{l-1}^{B_j}(x) + (b_l - b_{l-1})(B_{j-1}) \cdot ( B_{j-1} - x )^+ \leq c_l^{B_j}(x)  \qquad &&\forall x,
\label{eq:cond_d} \\
& \sum_{j=1}^{m} \left( c_{l}^{B_j}(x) - b_l(B_j) ( B_j - x )^+ \right) + c_{l}^{B_{m+1}}(x) = c_{\mu_l}(x) \quad &&\forall x.
\label{eq:cond_e} 
\end{alignat}
\end{subequations}
\end{itemize}

This market model can be chosen with bounded support (but does not have to).

Furthermore, in the market models characterized by \eqref{eq:cond_a}--\eqref{eq:cond_e} we have for $0 \leq x < B_j$, 
\begin{align}
\Prob{ S_{T_l} > x, M_{T_l} < B_j } = 
- \frac{\dd \mathfrak{c}^{B_j}_{l}}{\dd x}(x+) - b_l(B_j)
\label{eq:joint_prob_mm}
\end{align}
where
\begin{align}
\mathfrak{c}^{B_j}_{l}(x) := \sum_{i=1}^{j-1} \Big( c_l^{B_i}(x) - b_l(B_i) \cdot ( B_i - x )^+ \Big) + c_l^{B_j}(x).
\label{eq:definition_identification_1M}
\end{align}
\end{Theorem}

\begin{Remark}[Connection to Theorem \ref{thm:Characterization Market Model - One Maturity}]
\label{rem:connection_thms}
Identifying for $j=1,\dots,m+1$,
\begin{align}
\mathfrak{c}^{B_j}(x) = \mathfrak{c}^{B_j}_{1}(x) 
\end{align}
easily shows that Theorem \ref{thm:Characterization Market Model - Multiple Maturity} in the case $l=1$ implies Theorem \ref{thm:Characterization Market Model - One Maturity}.

Conversely, identifying (using the notation of Theorems \ref{thm:Characterization Market Model - One Maturity} and \ref{thm:Characterization Market Model - Multiple Maturity})
\begin{align}
c_1^{B_j}(x) = \mathfrak{c}^{B_j}(x) - \mathfrak{c}^{B_{j-1}}(x) + b(B_{j-1})(B_{j-1} - x)^+
\end{align}
for $j=1,\dots,m+1$, with the convention $\mathfrak{c}^{B_0}(x) := (B_0 - x)^+$,
easily shows that Theorem \ref{thm:Characterization Market Model - One Maturity} implies Theorem \ref{thm:Characterization Market Model - Multiple Maturity} in the case $l=1$.

Therefore, Theorem \ref{thm:Characterization Market Model - Multiple Maturity} provides a slightly different alternative characterization than Theorem \ref{thm:Characterization Market Model - One Maturity}.
\end{Remark}

\begin{proof}[Proof of Theorem \ref{thm:Characterization Market Model - Multiple Maturity}]
Let $(\Omega, \mathcal{F}, \F, \PP)$ be a market model, i.e. for $l=1,\dots,k$,
\begin{alignat*}{5}
\E{ \left( S_{T_l} - K_i \right)^+ } &= c_{\mu_l}(K_i) \qquad &&\forall && i&&=1,\dots,n, \\
\Prob{ M_{T_l} \geq B_j } &= b_l(B_j) \qquad &&\forall && j&&=1,\dots,m.
\end{alignat*}
By continuity of $S$ the Dambis-Dubins-Schwarz time change yields
\begin{align*}
\left( S_t \right)_{t \leq T_k} = \left( X_{\rho_t} \right)_{ t \leq T_k }
\end{align*}
where $X$ is Geometric Brownian motion started at $S_0$ and $\rho_t = \langle \log(S) \rangle_t$.

Define for $j=1,\dots,m+1, \ l=1,\dots,k$ and $x \in \R$, 
\begin{align}
&c_l^{B_j}(x) := \E{ \indicator{ \rho_{T_l} \geq H_{ B_{j-1} } } \left( X_{ \rho_{T_l} \wedge H_{B_j} } - x \right)^+ }.
\end{align}
Similarly as before, properties \eqref{eq:cond_a}--\eqref{eq:cond_c} are easily verified.
We compute
\begin{align*}
c_l^{B_j}(x) - c_{l-1}^{B_j}(x) = \E{ \indicator{ \rho_{T_{l-1}} < H_{ B_{j-1} }, \rho_{T_l} \geq H_{ B_{j-1} } } \left( X_{ \rho_{T_l} \wedge H_{B_j} } - x \right)^+ }
\end{align*}
which defines a measure with mass $(b_l - b_{l-1})(B_{j-1})$ and mean $ (b_l - b_{l-1})(B_{j-1}) \cdot B_{j-1} $ and hence \eqref{eq:cond_d} follows.
As for \eqref{eq:cond_e} we note that 
\begin{alignat*}{3}
& &&\sum_{j=1}^{m} \left( c_{l}^{B_j}(x) - b_l(B_j) ( B_j - x )^+ \right) + c_{l}^{B_{m+1}}(x) \\
=& &&\sum_{j=1}^{m} \E{ \indicator{ \rho_{T_l} \geq H_{ B_{j-1} } } \left( X_{ \rho_{T_l} \wedge H_{B_j} } - x \right)^+ - \indicator{ \rho_{T_l} \geq H_{B_j} } \left( B_j - x \right)^+ } \\
& \ \ + \ \ && \E{ \indicator{ \rho_{T_l} \geq H_{ B_{m} } } \left( X_{ \rho_{T_l} } - x \right)^+} \\
=& &&\sum_{j=1}^{m}\E{ \indicator{ H_{B_{j-1}} \leq \rho_{T_l} < H_{B_j} } \left( X_{\rho_{T_l}} - x \right)^+ } + \E{ \indicator{ \rho_{T_l} \geq H_{ B_{m} } } \left( X_{ \rho_{T_l} } - x \right)^+} \\
=& &&  c_{\mu_l}(x).
\end{alignat*}

Now assume that conditions \eqref{eq:cond_a}--\eqref{eq:cond_e} hold.
By Remark \ref{rem:connection_thms} the claim is true for $k=1$.
Inductively, let us assume that the claim is true until $k-1$, in particular there exist $\eta_l, \ l=1\dots,k-1,$ such that
\begin{subequations}
\begin{align}
c_{l}^{B_j}(x) &= \E{ \indicator{ \eta_{l} \geq H_{B_{j-1}} } \left( X_{ \eta_{l} \wedge H_{B_j} } - x \right)^+ }, \qquad x \in \R, 
\label{eq:quantities_induction_a} \\
b_{l}(B_j) &= \Prob{ M_{\eta_{l}} \geq B_j } = \Prob{ \eta_{l} \geq H_{B_j} }.
\label{eq:quantities_induction_b}
\end{align}
\end{subequations}
By \eqref{eq:cond_a} the two functions
\begin{align*}
&x \mapsto s_k^{B_j}(x) := c_{k-1}^{B_j}(x) + (b_k - b_{k-1})(B_{j-1}) \cdot ( B_{j-1} - x )^+, \\
&x \mapsto c_k^{B_j}(x),
\end{align*}
are convex and by \eqref{eq:cond_d} we have 
\begin{align*}
0 \leq s_k^{B_j} \leq c_k^{B_j}.
\end{align*}
By \eqref{eq:cond_b} the means of the measures corresponding to $s_k^{B_j}$ and $c_k^{B_j}$ coincide, 
\begin{align*}
s_k^{B_j}(0) = b_{k-1}(B_{j-1}) \cdot B_{j-1} + (b_{k}-b_{k-1})(B_{j-1}) \cdot B_{j-1} = b_k(B_{j-1}) \cdot B_{j-1} = c_k^{B_j}(0),
\end{align*}
and the same is true for the masses, 
\begin{align*}
\frac{\partial s_k^{B_j}}{\partial x}(0+) = b_{k-1}(B_{j-1}) + (b_{k} - b_{k-1})(B_{j-1}) = b_k(B_{j-1}) = \frac{\partial c_k^{B_j}}{\partial x}(0+).
\end{align*}
Hence, by \citet{Strassen65} there exists a stopping time $\varrho_k^{B_j}$ which embeds the measure corresponding to $s^{B_j}_k$ into the measure corresponding to $c_k^{B_j}$. 

We define recursively
\begin{subequations}
\begin{align}
\eta_k \wedge H_{B_1} &= 
\begin{cases} 
	\varrho_k^{B_1} &\mbox{if } \eta_{k-1} < H_{B_1}, \\
	H_{B_1} & \mbox{else}, 
\end{cases} 
\label{eq:def_embedding_1} \\
\eta_k \wedge H_{B_2} &= 
\begin{cases} 
	\eta_{k} \wedge H_{B_1}  &\mbox{if } \eta_k \wedge H_{B_1} < H_{B_1},  \\
	\eta_{k} \wedge H_{B_1} + \varrho_k^{B_2} \circ \theta_{ \eta_{k} \wedge H_{B_1} } &\mbox{if } \eta_{k-1} < H_{B_2}, \ \eta_k \wedge H_{B_1} = H_{B_1}, \\
	H_{B_2} & \mbox{else}, 
\end{cases} 
\label{eq:def_embedding_2} \\
&\hspace{2mm}\vdots  \nonumber\\
\eta_k \wedge H_{B_m} &= 
\begin{cases} 
	\eta_{k} \wedge H_{B_{m-1}}  &\mbox{if } \eta_k \wedge H_{B_{m-1}} < H_{B_{m-1}},  \\
	\eta_{k} \wedge H_{B_{m-1}} + \varrho_k^{B_m} \circ \theta_{ \eta_{k} \wedge H_{B_{m-1}} } &\mbox{if } \eta_{k-1} < H_{B_m}, \ \eta_k \wedge H_{B_{m-1}} = H_{B_{m-1}}, \\
	H_{B_m} & \mbox{else}, 
\end{cases} 
\label{eq:def_embedding_m} \\
\eta_k &= 
\begin{cases} 
	\eta_{k} \wedge H_{B_{m}}  &\mbox{if } \eta_k \wedge H_{B_{m}} < H_{B_{m}}, \\
	\eta_{k} \wedge H_{B_{m}} + \varrho_k^{B_{m+1}} \circ \theta_{\eta_{k} \wedge H_{B_{m}}} &\mbox{if } \eta_k \wedge H_{B_{m}} = H_{B_{m}}.
\end{cases} 
\label{eq:def_embedding_final}
\end{align}
\end{subequations}

Next we show that this construction recovers the correct quantities. 
It is already visible from \eqref{eq:def_embedding_1}--\eqref{eq:def_embedding_final} that the condition for $\varrho_k^{B_j}$ is triggered only for paths which hit the level $B_{j-1}$ and hence $\varrho_k^{B_j}$ does not change the probability that $B_i,i \leq j-1$, is hit. However, it can change the probability that $B_j$ is hit.

Denote $\EE^{s_k^{B_j}}\left[ \left( X_{\gamma} - x \right)^+ \right]$ the expectation of $X_{\gamma}$ where $X_0 \sim \nu_k^{B_j}$ and $\nu_k^{B_j}$ is the measure corresponding to $s_k^{B_j}$.
With this definition it follows inductively that for $j=1,\dots,m+1$,
\begin{alignat*}{3}
& \ \ && \E{ \indicator{ \eta_{k} \geq H_{B_{j-1}} } \left( X_{ \eta_{k} \wedge H_{B_j} } - x \right)^+ } \\
= & && \E{ \left( \indicator{ \eta_{k-1} \geq H_{B_{j-1}} } + \indicator{ \eta_{k-1} < H_{B_{j-1}}, \eta_{k} \wedge H_{B_{j-1}} = H_{B_{j-1}} } \right) \left( X_{ \eta_{k} \wedge H_{B_j} } - x \right)^+ } \\
= & && \E{ \indicator{ \eta_{k-1} \geq H_{B_{j-1}} } \left( X_{ \left( \eta_{k-1} + \varrho_k^{B_j} \circ \theta_{\eta_{k-1}} \right) \wedge H_{B_j} } - x \right)^+ } \\
 & \ + \ && \E{ \indicator{ \eta_{k-1} < H_{B_{j-1}}, \eta_{k} \wedge H_{B_{j-1}} = H_{B_{j-1}} } \left( X_{ \left( H_{B_{j-1}} + \varrho_k^{B_j} \circ \theta_{H_{B_{j-1}}} \right) \wedge H_{B_j} } - x \right)^+ } \\
= & && \EE^{ s_k^{B_j} } \left[ \left( X_{\varrho_k^{B_j}} - x \right)^+ \right] = c_{k}^{B_j}(x)
\end{alignat*}
and therefore
\begin{align*}
&\Prob{ \eta_k \geq H_{B_j} } = -\frac{\partial c_k^{B_j}}{\partial x}(B_j-) \stackrel[]{\eqref{eq:cond_c}}{=} b_k(B_j), \\
&\Prob{ \eta_k \wedge H_{B_j} \geq H_{B_i} } = 
\Prob{ \eta_k \wedge H_{B_i} \geq H_{B_i} } = b_k(B_i), \qquad \forall i<j.
\end{align*}
Finally, the embedding property follows as
\begin{alignat*}{3}
& && \E{ \left( X_{\eta_k} - x \right)^+ } \\
=& && \sum_{j=1}^{m} \E{ \indicator{ \eta_k \in \big[ H_{B_{j-1} } , H_{B_j} \big) } \left( X_{\eta_k } - x \right)^+ } + \E{ \indicator{ \eta_k \geq H_{B_{m} } } \left( X_{\eta_k} - x \right)^+ }   \\
=& && \sum_{j=1}^{m} \E{ \indicator{ \eta_k \geq H_{B_{j-1} } } \left( X_{\eta_k \wedge H_{B_j}} - x \right)^+ } - \Prob{ \eta_k \geq H_{B_j} } (B_j - x )^+  \\
& \ \ + \ \ && \E{ \indicator{ \eta_k \geq H_{B_{m} } } \left( X_{\eta_k} - x \right)^+ }   \\
=& &&   \sum_{j=1}^{m} \left( c_{k}^{B_j}(x) - b_{k}(B_j) ( B_j - x )^+ \right) + c_{k}^{B_{m+1}}(x) \stackrel[]{\eqref{eq:cond_e}}{=} c_{\mu_k}(x).
\end{alignat*}
 
The claim regarding the existence of a market model with a bounded support and equation \eqref{eq:joint_prob_mm} follow in the same way as in the proof of Theorem \ref{thm:Characterization Market Model - One Maturity}.
\end{proof}

\section{Interpolation}
\label{sec:interpolation}

Theorems \ref{thm:Characterization Market Model - One Maturity} and \ref{thm:Characterization Market Model - Multiple Maturity} require to \enquote{decompose} some call price functions into a sequence of intermediate convex functions. 
Given this decomposition, equations \eqref{eq:joint_prob} and \eqref{eq:joint_prob_mm} partially specify the joint marginal distributions of $(S,M)$ implied by this decomposition.
Next we discuss how to consistently interpolate these joint probabilities from \eqref{eq:joint_prob} and \eqref{eq:joint_prob_mm} in barriers and time.

\subsection{Computation of Decomposition}

The strength of our main result hinges on the computation of the quantities $\left\{ c^{B_j}_{l} \right\}$ as described in Theorem \ref{thm:Characterization Market Model - Multiple Maturity}.
In practice, a simple and efficient method to compute them would be to discretize in space and solve a linear programming (LP) problem. 
If $N$ is the number of discretization points in space, the number of variables is $\mathcal{O}(N)$ in the one maturity case.
Of course, a naive linear program which optimizes over the joint distribution of maximum and terminal value is also possible, but it has more variables: in the one maturity case it is of order $\mathcal{O}(N^2)$.
In addition, one would also need to make sure that one obtains a valid joint distribution for the stock \emph{and} its maximum, see conditions \eqref{eq:rogers1}--\eqref{eq:rogers2} below.
In the case of $k$ maturities, the number of variables is $\mathcal{O}(k N)$ to compute  $\left\{ c^{B_j}_{l} \right\}$ as described in Theorem \ref{thm:Characterization Market Model - Multiple Maturity}.
When trying to use a naive linear program one would need to ensure several condition regarding the \textit{ordering} of the joint distributions, see Section \ref{subsec:Interpolation in Time} below.
This might not be straightforward to implement.

By changing the objective function of the LP one is able to achieve two things.
Firstly, one can regularize the solution by \enquote{penalizing} e.g. gaps in support or atoms.
Secondly, one can find solutions with additional features such as maximizing the expectation of a given payoff, which would yield a upper robust price bound for this payoff. In particular, one could calculate the maximal price of a simple barrier option with barrier $B \neq B_j$.

\subsection{Interpolation of Barrier Prices}
\label{subsec:Interpolation of Barrier Prices}

In the models characterized by the decomposition $\left\{ c^{B_j}_{l} \right\}$, equation \eqref{eq:joint_prob} of Theorem \ref{thm:Characterization Market Model - One Maturity} partially specifies the joint (tail) distribution for $(S_T,M_T)$ as 
\begin{align}
\bar{F}_T(x,B_j) := \Prob{ S_T > x, M_T \geq B_j } = -\frac{\dd c_{\mu}}{\dd x}(x+) + \left( \frac{\dd \mathfrak{c}^{B_j}}{\dd x}(x+) + b(B_j) \right) \indicator{ x<B_j }
\end{align}
for $x \geq 0$ and $j=0,\dots,m$.

Recall from Theorem \ref{thm:Characterization Market Model - One Maturity} that in order to incorporate a bounded support for the joint distribution we can impose that $B_m = K_n$ and $b(B_m) = 0 = c_{\mu}(K_n)$ for $B_m$ sufficiently large.

In order to obtain an unbounded support, we have to extrapolate.
In this context, note that Theorem \ref{thm:Characterization Market Model - One Maturity} is readily extended to countably many call and barrier options with increasing strikes and barriers.

\citet[Theorem 2.2]{Rogers:93} characterizes the set of all these possible distributions by some integrability condition and two properties of the function
\begin{align}
d(m) = \E{ S_T \big| M_T > m },
\end{align}
namely
\begin{subequations}
\begin{align}
&d(m)\geq m, \label{eq:rogers1}\\
&m \mapsto d(m) \qquad \text{is non-decreasing}. \label{eq:rogers2} 
\end{align}
\end{subequations}
By construction we have for $j=1,\dots,m$,
\begin{subequations}
\begin{align}
&d( B_j ) \geq B_j, \\
&d( B_{j-1} ) \leq d( B_{j} ). 
\end{align}
\end{subequations}
The simplest interpolation is to use linear interpolation in barrier option prices. 
Let $a \in (0,1)$ and $B=a B_{j-1} + (1-a)B_{j}$.
Setting 
\begin{align}
\bar{F}_T(x,B) := a \bar{F}_T(x,B_{j-1}) + (1-a)\bar{F}_T(x,B_{j})
\end{align}
yields a joint (tail) distribution $\bar{F}_T$ which satisfies \eqref{eq:rogers1}--\eqref{eq:rogers2}.


We will refer by $\mathcal{M}_{T}$ to the measure corresponding to $\bar{F}_T$. 

\subsection{Interpolation in Time}
\label{subsec:Interpolation in Time}

An interpolation in time is not as easily obtained because the interpolations for fixed maturities cannot be done independently of each other.

\citet[Theorem 4]{Rost71} characterizes when there exists a martingale $S$ such that
\begin{align*}
(S_{T_1}, M_{T_1}) \sim \mathcal{M}_{T_1}, \qquad (S_{T_2}, M_{T_2}) \sim \mathcal{M}_{T_2}. 
\end{align*}
However, his characterization is not very explicit in our setup.

To see in a simple example that things can indeed go wrong, consider 
\begin{align*}
\delta_{ \{ S_0 \} } \preccurlyeq_{\mathrm{c}} \mu_1 \preccurlyeq_{\mathrm{c}} \mu_2, 
\quad b_{\mu_1} \nleq b_{\mu_2} 
\quad \mathcal{M}_{T_1} := \mathcal{L}\left( X_{\tau^{\mathrm{AY}}_{\mu_1}}, M_{\tau^{\mathrm{AY}}_{\mu_1}} \right), 
\quad \mathcal{M}_{T_2} := \mathcal{L}\left( X_{\tau^{\mathrm{AY}}_{\mu_2}}, M_{\tau^{\mathrm{AY}}_{\mu_2}} \right)
\end{align*}
where $b_{\mu}$ denotes the barycenter function of $\mu$ and $\tau^{\mathrm{AY}}_{\mu}$ denoted the \AY embedding of $\mu$.
It is known, see e.g. \citet{Rogers:93}, that both $\mathcal{M}_{T_1}$ and $\mathcal{M}_{T_2}$ can be embedded starting from the Dirac measure $\delta_{ \{ S_0 \} }$
and that $M_{\tau^{\mathrm{AY}}_{\mu_1}} \sim \mu_1^{\mathrm{HL}}$ and $M_{\tau^{\mathrm{AY}}_{\mu_2}} \sim \mu_2^{\mathrm{HL}}$, respectively, where $\mu^{\mathrm{HL}}$ denotes the Hardy-Littlewood transform of $\mu$.
However, it follows from \citet{Brown98themaximum} that it is not possible to embed $\mathcal{M}_{T_2}$ after $\mathcal{M}_{T_1}$ because of $b_{\mu_1} \nleq b_{\mu_2}$.

Therefore, if the interpolation method from Section \ref{subsec:Interpolation of Barrier Prices} yield $\mathcal{M}_{T_1}$ and $\mathcal{M}_{T_2}$ as above, then it follows that this interpolation is inconsistent with a model.

\subsection{Interpolation via Skorokhod Embedding} 

One theoretical way to specify the joint laws $\left( \mathcal{M}_t \right)_{t \leq T}$ is to use one's favourite Skorokhod embedding $(\tau_t)_{t \leq T}$ as described in the proofs of Theorems \ref{thm:Characterization Market Model - One Maturity} and \ref{thm:Characterization Market Model - Multiple Maturity}.
The functions $c_l^{B_j}$ can be interpreted as the footprint of the marginal law evolution of the process. 
This yields marginal laws as
\begin{align}
\mathcal{M}_t := \mathcal{L}\left( X_{\tau_t}, M_{\tau_t} \right) \qquad \forall t \leq T.
\end{align}

\bibliographystyle{plainnat}
\bibliography{literature_papers3}

\end{document}